\documentclass{amsart}

\usepackage{lineno,hyperref}

\usepackage{graphicx}
\usepackage{amssymb}
\usepackage[utf8]{inputenc}     
\usepackage{comment}
\usepackage{amsfonts}
\usepackage{amsthm}
\usepackage{amssymb}
\usepackage{amsmath}
\usepackage{enumerate}
\usepackage{mathtools}
\usepackage{tikz}
\usepackage{multirow}
\usepackage{color}
\newtheorem{theorem}{Theorem}

\newtheorem{lemma}[theorem]{Lemma}

\newtheorem{proposition}[theorem]{Proposition}
\newtheorem{corollary}[theorem]{Corollary}
\everymath{\displaystyle }
\allowdisplaybreaks

\title{On the number of  words \\with restrictions on the  number of symbols}
\thanks{The authors are members of LIA SINFIN (ex-INFINIS), 
Universit\'e de Paris-CNRS/Universidad de Buenos Aires-CONICET
and they are part of the   STIC AMSUD project 20STIC-06.
Becher is supported  by  PICT2018-02315 and Cesaratto  is partially supported by Grant UNGS 30/3307.}

\author{Ver\'onica Becher}
\address{V. Becher \\ Departamento de  Computaci\'on, Facultad de Ciencias Exactas y Naturales \& ICC  \\
 Universidad de Buenos Aires \&  CONICET, Argentina}
\email{vbecher@dc.uba.ar}
\author{Eda Cesaratto}
 
\address{E. Cesaratto\\ Universidad Nac. de Gral. Sarmiento \& CONICET, Argentina}
\email{ecesaratto@campus.ungs.edu.ar}

\begin{document}

\begin{abstract}
We show that, in an alphabet of $n$ symbols,  the number of words  of length~$n$ 
whose  number of different symbols is away from $(1-1/e)n$,
which is the value expected  by the Poisson distribution,  has  exponential decay in $n$.
We use   Laplace's method for sums and
 known bounds of Stirling numbers of the second kind.
We express our result in terms of inequalities.
\end{abstract}

 \maketitle

\noindent
{\bf MSC 2020}: 05A05, 05A10, 05A20
\bigskip

\noindent
{\bf Keywords}:  Poisson distribution; Laplace method for sums; Stirling numbers of the second kind; Combinatorics on words.
\bigskip

\section{Introduction and statement of results}

Consider an alphabet of $n$ symbols and 
let $\chi^{(i)}$ be the number of symbols  that appear exactly~$i$ times in a word of length $m$. This can be seen as the allocation of $m$ balls (the  positions in a word of length $m$) in $n$ bins (the $n$~symbols of the alphabet), which determines a total  of   $n^{m}$  allocations. When  $m/n$  is a fixed constant~$\lambda$,
\[
\frac{1}{n}\chi^{(i)} \text{ converges in probability to }e^{-\lambda} \frac{\lambda^i}{i!},
\]
which is the Poisson formula, the proof can be read from~\cite[Example III.10 and  Proposition~V.11]{Flajolet&Sedgewick:2009}. 

We are interested in the case when the alphabet size $n$ equals the word length $m$, hence  $\lambda=m/n=1$.
The number  of symbols that do {\em not} appear in a  word of length $n$ is  $\chi^{(0)}$ and 
its expected value is~$n/e$.  
Hence, the  expected  number of different symbols in a word of length~$n$ is $n-n/e=(1-1/e)n$. 
The probability that   $\chi^{(0)}$ is equal to  $j$ for $j=0, 1, \ldots, n$  is expressible in terms of Stirling 
numbers of the second kind:  the number  $a(n,j)$ of words of length $n$ with {\em exactly}~$j$ different symbols
 is the  number of ways to choose $j$ out of $n$ elements  times  the number 
 of surjective maps from a set of $n$  positions to a set of $j$  symbols.
To  make such a surjective map,  first partition the set of $n$ elements  into $j$ 
nonempty subsets and, in one of  the $j!$ many ways, 
assign one of these subsets to  each element  in the set of~$j$ elements,
\[
a(n,j)=   
\binom{n}{j} j!\  S_n^{(j)},
\]
where 
 \[
S_n^{(j)}=\frac{1}{j!} \sum_{i=0}^{j} (-1)^i  
\binom{j}{i}
(j-i)^n.
\]
Notice that 
\[
\sum_{j=0}^{n}  a(n,j)= n^n. 
\]

Theorem~\ref{thm:a(n,j)}  is  the main result of this note and 
shows that in an alphabet of $n$ symbols,  
the number of words of length $n$ with exactly $j$ symbols, 
has exponential decay in $n$ 
when~$j$ is away from the value expected by the Poisson distribution.
Precisely, Theorem~\ref{thm:a(n,j)} proves that 
$a(n,j)$,  has exponential decay in $n$ 
 when $j$ is away from~$(1-1/e)n$. 
And this implies  that  for every positive~$\varepsilon<1$,
\[
\sum_{n\geq 1}   n^{-n} \left(\sum_{j=1}^{(1-1/e - \varepsilon)n} a(n,j) + \sum_{j=(1-1/e + \varepsilon)n}^{n} a(n,j)\right)< \infty.
\]


\begin{theorem}\label{thm:a(n,j)} 
There is a function $\phi: (0,1)\mapsto \mathbb R$ such that 
$\phi(x)<1$ for every $x\not=1-1/e$,
 positive reals  $r$ and $\Lambda$ both less than $1$, 
and positive constants $c$ and $C$
 satisfying the following condition: 
\nopagebreak
\noindent 
For every pair $n,j$ of  integers with  $1\le j\le n$, 
\begin{align*}
a(n,j) &\le 
\begin{cases} C \sqrt{n} \Lambda^n n^{n}  \ \ \ & \hbox{, if }
j/n\in [0,r]\cup [1-r,1] 
\\ 
                      C\ \phi(j/n)^n        n^{n}    \ \ \  & \hbox{, if }
                       j/n\in [r, 1-r]
                      \\
\end{cases}
\\a(n,j)  &\ge\  (c/ \sqrt{n})\, \phi(j/n)^n n^{n}\ \hbox{, if } 
j/n\in [r, 1-r].
\end{align*}
\nopagebreak
Precisely,
\begin{align*}
\phi:(0,1)\mapsto \mathbb R,  & \  \  \ \phi(x)=(e\ln(1+e^{-\delta(x)})^{-1}\varphi(x)e^{-x\delta(x)}
\\
\varphi:[0,1]\mapsto \mathbb R,& \  \ \  \varphi(x)=x^{-x}(1-x)^{-(1-x)},\  \varphi(0)=\varphi(1)=1
\\
\delta:(0,1)\mapsto \mathbb R,&  \ \ \  \delta^{-1} (y) = \frac{1}{(1+e^{y})\ln(1+e^{-y})}.
\end{align*}
\end{theorem}

Each of the values $c,C,\Lambda$ and $r$ in the statement of Theorem~\ref{thm:a(n,j)}  can be effectively computed.
Figure~\ref{Fig:Thm1} plots the upper bound of $\sqrt[n]{a(n,j)} n^{-1}$  with  the function $\phi(j/n)$ given in Theorem~\ref{thm:a(n,j)}.
\medskip

As a straightforward application   of Theorem~\ref{thm:a(n,j)} we obtain the following.

\begin{figure}
\begin{center}
 \includegraphics{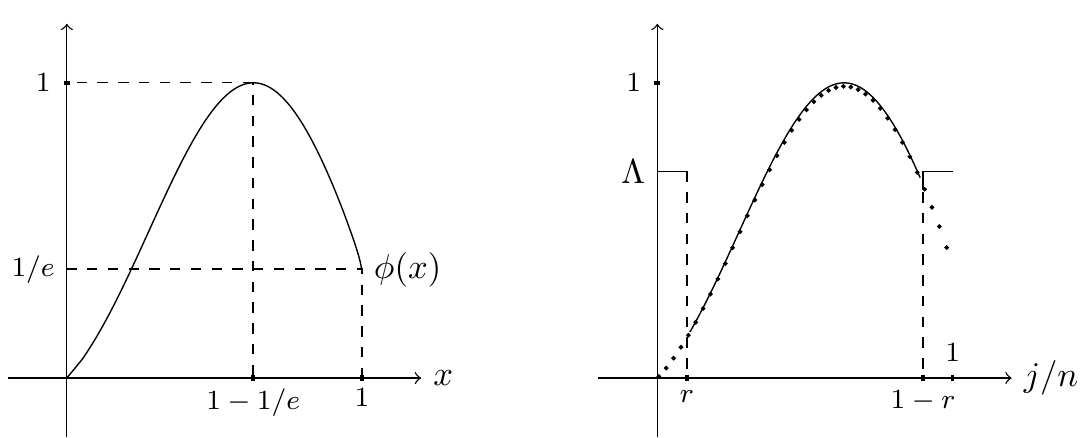}
\end{center}
\caption{On the left, the graph of $\phi(x)$. 
On the right,  the points are $\sqrt[n]{a(n,j)} n^{-1}$ for $n=200$ and $j=0,5,10,\dots,195,200$
 and  the solid line is $\phi(j/n)$  with $r=0.1$ and $\Lambda\approx 0.701$.} 
\label{Fig:Thm1}
\end{figure}

\begin{corollary}\label{cor:sum} 
For any positive real number $\varepsilon$ there exist positive constants $c$ and $C$ 
and a positive real number~$\Lambda $ strictly less than $1$ such that for every positive integers~$n,\ell$,
\[
\text{if } \ \  |\ell/n -( 1- 1/e)|\ge \varepsilon \ \  \text{ then }\ \ 
 (c/\sqrt{n})\ \Lambda^n  
\leq  n^{-n} \sum_{j=1}^{\ell} a(n,j) 
\leq  C n \sqrt{n} \ \Lambda^n.
\]
\end{corollary}

A  tail estimate is  a quantification of   the rate of decrease of probabilities away from the central part of a distribution.
It is known that the tail of   a given arbitrary  discrete distribution  has exponential decay  if   its probability generating function is analytic on a disk centered on zero and of radius
greater than $1$~\cite[Theorem IX.3, page 627]{Flajolet&Sedgewick:2009}.    Theorem~\ref{thm:a(n,j)} gives, indeed, a tail estimate with exponential decay, but our methods are not analytic. 

Our  proof  of Theorem~\ref{thm:a(n,j)} is elementary except for the estimates for Stirling numbers of the second kind that we use as a black box.
We follow the principles of  Laplace's method for sums,
which is useful for  sums of positive terms which increase to a certain point and then decrease.  
For a general explanation with examples  we refer to Flajolet and Sedgewick's 
book~\cite[p.761]{Flajolet&Sedgewick:2009}, 
see~\cite{Paris} for a rigorous application to an hypergeometric-type series. 
However, we do not use the exp-log transformation to build the approximation function.

Specifically, to prove Theorem~\ref{thm:a(n,j)} we give a smooth function $\phi$ so that $\phi(j/n)^n$ bounds    
$a(n,j)n^{-n}$ from above and below  
(up to  multiplicative  sequences that increase or decrease slowly). 
We consider  the ratio between $j$ and~$n$.
When $j$ is near  to~$0$  or near to $n$
we  use the classical 
upper bound of Stirling numbers of the second kind 
given by  Rennie and Dobson~\cite{Rennie}.
When  $j$ 
is not near to $0$ nor near  to $n$  we use Bender's 
approximation of Stirling numbers of the second kind~\cite{Bender} as a  black box.  
This approximation comes from analytic combinatorics methods and 
it was initially devised by  Laplace,  then  proved 
by Moser and Wyman~\cite{mw} and later sharpened by Bender, see also~\cite{l}. 
Our two choices are motivated by  the  comparison of bounds on Stirling numbers by Rennie-Dobson \cite{Rennie}, Arratia and DeSalvo \cite{ADS}, and also a trivial bound, given in Section~\ref{sec:comparison}.

The approach we use in the proof of Theorem~\ref{thm:a(n,j)}   was previously used 
by one of the authors in two different problems. 
In~\cite{Edatesis} it is used  to  estimate  $n!\prod_{i=1}^k p_i^{ j_is}/j_i!$ 
where each  $p_i$ is the probability of the symbol~$i$ in an alphabet of $k$ elements,  
$s$ is a real number in $(0,1)$
and the integers $j_i$ 
sum up $n$ and  $\sum_{i=1}^k i j_i\le M n$ for a fixed $M>1$. In~\cite[Remark 4.3]{CeMaPePri}  the  same approach is used to obtain an upper bound for   $ \binom{n}{j}/j!$ 
when $n$ is fixed and $j$~varies. 
Besides, the asymptotic behavior of these quantities  when $n$ tends to infinity
was studied using a similar technique  in~\cite{LiPi}.

We crossed the  problem  solved in the present note
when studying  the set $\mathcal S$  of  infinite binary sequences
with too many or too few, with respect to the 
expected by the Poisson distribution,
 different words of length  $\lfloor \log n \rfloor$, counted  with no overlapping
 in their initial 
segment of length $n \lfloor \log n \rfloor$, 
for infinitely many $n$s.
Corollary~\ref{cor:sum}   allows us to prove  that the 
Lebesgue measure of this set~${\mathcal S}$ is zero, as follows.
For simplicity,  
let~$n$ be a power of $2$ and  let $\log$ be the  logarithm in base~$2$.
Identify  the  binary words of length $\log n $
with   integers from  $0$ to $n-1$.
Thus,   each binary word of length $n \log n$
is identified with a with a word of  $n$ integers from  $0$ to $n-1$.
Notice that there  are $2^{n \log n } = n^n$ many of these  binary words.
Corollary~\ref{cor:sum}  assumes an alphabet of $n$ symbols
and gives an upper bound for the proportion
 of  words of length~$n$ having a number of different symbols away from  $(1-1/e)n$, 
which is the quantity expected by the Poisson distribution.
By the identification we made, this yields  
an upper bound of the  proportion of binary words of length
 $n\log n  $  having   too many or too few  different binary blocks 
with respect to what is  expected by the Poisson distribution.
Since this upper  bound has exponential decay in $n$, 
we can  apply Borel-Cantelli lemma to show that  the sum, for every $n$, 
of these bounds is finite.  
Consequently, the  Lebesgue measure of the set ${\mathcal S}$   is zero.
A different proof of this result  follows from the metric theorem  given by  Benjamin Weiss and Yuval Peres in~\cite{weiss2020} where they show that the set of Poisson generic  sequences on a finite alphabet  has  Lebesgue measure $1$. Their proof is probabilistic, with   a randomized part and a concentration part.

\section{On different bounds on Stirling numbers of second kind} \label{sec:comparison}

We compare four estimates on Stirling numbers of the second kind $S_n^{(j)}$. When $j/n$ belongs to~$(0,1)$, we consider   a trivial bound, 
Rennie and Dobson's bound~\cite{Rennie}  and 
Arratia and DeSalvo's bounds~\cite{ADS}. When $j/n$ belongs to a closed interval included in $(0,1)$, we consider   Bender's estimate~\cite{Bender}.
We start by giving  bounds for the   binomial coefficients.

\subsection{Binomial coefficients}
Consider the following bounds for the factorial   which are consequence of the classical Stirling's formula for the factorial, 
see~\cite{Robbins},
\[
n!=\sqrt{2\pi}n^{n+1/2}e^{-n+r_n}, \ \ \ 
\frac{1}{12n+1}\le r_n\le \frac{1}{12n}.
\]
Then, for any $n\ge 1$, 
\begin{equation}\label{StirlingRobbins}
\sqrt{2\pi}n^{n+1/2}e^{-n}\le n!\le \sqrt{2\pi}e^{1/12}n^{n+1/2}e^{-n} .
\end{equation}

In the sequel we write $a \approx b $ to indicate  that  the two numbers $a$ and $b$ coincide up to the precision  explicitly indicated, 
but they may differ in the fractional part that is not exhibited. For example, $\pi\approx 3.14159$.
From this approximation of the factorial, we obtain  bounds for the binomial coefficient that involve  the following functions,
\begin{align}\label{combinatorial}
 \varphi:[0,1]\mapsto \mathbb R, & \quad \varphi(x)=x^{-x}(1-x)^{-(1-x)},\  \varphi(0)=\varphi(1)=1;  \\
\gamma:(0,1)\mapsto \mathbb R, & \quad \gamma(x)=(x-x^2)^{1/2}  \nonumber
\end{align}

There exist constants $c_0$ and $C_0$ such that  for any pair of integers $n,j$ where  $n\ge 2$ and $1\le j\le n-1$,  
\[
\frac{c_0}{\sqrt{n}\gamma(j/n)}\varphi(j/n)^n\le \binom{n}{j}\le \frac{C_0}{\sqrt{n}\gamma(j/n)}\varphi(j/n)^n\ . 
 \]
The constants $c_0$ and $C_0$ can be chosen as  $c_0=(\sqrt{2\pi}e^{1/6})^{-1}\approx 0.33$ and $C_0=e^{1/12}(\sqrt{2\pi})^{-1}\approx 0.43 $.
From~\eqref{StirlingRobbins}, it follows that  
\begin{equation}\label{Eq:binup}
\binom{n}{j}\le e^{1/12}(\sqrt{2\pi})^{-1}\left(\frac{n }{j(n-j)}\right)^{1/2}\frac{n^n}{j^j(n-j)^{n-j}}. 
\end{equation}
First, notice that 
\[\left(\frac{n }{j(n-j)}\right)^{1/2}= \frac{n^ {1/2}}{{n}(j/n (1-j/n))^{1/2}}=\frac{1}{\sqrt{n}\gamma(j/n)}.
\]
Now, we deal with the last factor of~\eqref{Eq:binup}. The following holds:
\begin{align*}
\frac{n^n}{j^j(n-j)^{n-j}}&=\frac{n^n}{n^n(j/n)^{j}(1-j/n)^{n-j}}\\
                           &=\left({(j/n)^{-j/n}(1-j/n)^{-(1-j/n)}}\right)^{n}\\
                           &=\varphi(j/n)^{n}.
\end{align*}
This proves the upper bound on the binomial coefficient. The proof of the  lower  bound is similar, except that  the factor    $e^{1/12}$ appears twice in the denominator.

Finally, we remark that for any pair of  positive integers $n,j$ such that  $n\ge 2$  and  $1\le j\le n-1$, we have 
$\min\{j(n-j): 1\le j\le n-1\}=n-1$ (this value is attained at  $j=1$ or $j=n-1$). Also   $n-1\ge n/ 2$ for $n\ge 2$. Hence, 
\begin{align*}\gamma(j/n)&= \left(\frac{j(n-j)}{n^2}\right)^{1/2}\ge \left(\frac{n}{2n^2}\right)^{1/2}=\frac{\sqrt{2}}{2}n^{-1/2}, \text{ and }\\
\gamma(j/n)&\le \max\{ {\gamma(x): x\in [0,1]}\}= \max\{ {(x-x^2)^{1/2}: x\in [0,1]}\}\le 1/2.
\end{align*}
Thus, multiplying by $\sqrt n$,
\[
   {\sqrt{2}}/2\le \sqrt{n}\gamma(j/n)\le (1/2)\sqrt{n}, 
 \]
which implies 
 \[
   \frac{2}{\sqrt n}\le \frac1{\sqrt{n}\gamma(j/n)}\le {\sqrt 2}.  \]
We have that $2c_0 \approx 0.67> 1/2$ and $\sqrt 2 C_0\approx 0.61<1$. 
This shows {the following  inequalities}, for every positive  $n\ge 2$  and  every $j$ such that  $1\le j\le n-1$,
\begin{equation}\label{binom}
\frac{1}{{2}\sqrt{n}}\varphi(j/n)^n\le  \frac{2c_0}{\sqrt n} \varphi(j/n)^n \le   \binom{n}{j}\le \sqrt{2}C_0 \varphi(j/n)^n\le \varphi(j/n)^n .
\end{equation}

\subsection{A trivial bound on Stirling numbers the second kind}\label{s:trivial}

The simplest upper bound 
takes just the  first term  of the alternating sum that defines $S_n^{(j)}$, \[
S_n^{(j)}\le j^n/j!.
\]
This upper bound appears explicitly taking  just  one term in Bonferroni inequalities,
 see~\cite[Section 4.7]{comtet}. 
First remark that the upper bound  given in ~\eqref{StirlingRobbins} for the factorial  yields 
\begin{equation}
\label{tb}
\frac{j^n}{j!}\le \frac{j^n}{\sqrt{2\pi j}j^je^{-j}}=\frac{1}{\sqrt{2\pi j}} \frac{n^{n-j}(j/n)^ne^{j}}{(j/n)^j } = \frac{1}{\sqrt{2\pi j}} \left(n^{1-j/n}(j/n)^{1-j/n}e^{j/n} \right)^n.
\end{equation}
The same lines together with the lower bound for~\eqref{StirlingRobbins} give a lower bound for~$j^n/j!$.

Let $\theta:[0,1]\mapsto \mathbb R$, 
\begin{equation}\label{theta}
\theta(x)=x^{1-x}e^x.
\end{equation}
It follows that  
\[
 \frac{1}{e^{1/12}\sqrt{2\pi j}}\left(n^{1-j/n}\theta(j/n)\right)^n\le 
 j^n/j!\le \frac{1}{\sqrt{2\pi j}}\left(n^{1-j/n}\theta(j/n)\right)^n. 
\]
Consequently,
\begin{equation}\label{trivial}
 S_n^{(j)}\le \frac{1}{\sqrt{2\pi j}}\left(n^{1-j/n}\theta(j/n)\right)^n.
\end{equation}

\subsection{Rennie and Dobson's bound}
The following  is the classsical upper bound of  Stirling numbers of the second kind  
given by Rennie and Dobson~\cite{Rennie},
which holds for every positive $n$ and  every $j$ such that $1\le j\le n-1$,
\begin{equation}\label{Rennie}
 S_{n}^{(j)}\le \frac{1}{2} \binom{n}{j} j^{n-j}.
\end{equation}

Let $\eta:[0,1]\mapsto \mathbb{R}$, 
\begin{equation}\label{eta}
\eta(x)=x^{1-x}\varphi(x),
\end{equation}
where $\varphi$ is defined in~\eqref{combinatorial}.
Since $j^{n-j}=(n^{1-j/n}(j/n)^{1-j/n})^n$, the bounds on the binomial given in~\eqref{binom} imply
\begin{equation}\label{Rennie-grafico}
 \frac{1}{{2}\sqrt n} \left(n^{1-j/n}\eta(j/n)\right)^n\le \binom{n}{j} j^{n-j}\le \left(n^{1-j/n}\eta(j/n)\right)^n.
\end{equation}

\subsection{Arratia and DeSalvo's bound}\label{s:Arratia}

Arratia and DeSalvo~\cite[Theorems 5 and 6]{ADS}
give these bounds  for  $n\ge 3$ and $1\le j \le n-2$,
\begin{align*}
 S_n^{(j)}\le& A_5(n,j)
 \\
S_n^{(j)}\le& A_6(n,j)
\end{align*}
where
\begin{align*}\label{A5}
A_5(n,j):=&\binom{N}{n-j}e^{-2\mu_5(n,j)}\left(1+e^{2\mu_5(n,j)}D_5(n,j)\right)
\\
A_6(n,j):=&\frac{N^{n-j}}{(n-j)!}e^{-\mu_6(n,j)}\left(1+e^{\mu_6(n,j)}D_6(n,j)\right)
\\
N:=&\binom{n}{2}
\\
 \mu_5(n,j):=&\binom{(n-j)}{2}\binom{n}{3}/ \binom{N}{2}
\\
\mu_6(n,j):=& \binom{(n-j)}{2} \frac{n(n-1)(4n-5)}{6 N^2}
\\
d_5(n,j):=&\ P+ Q+ (1- Q)((n-j)-2)(R+S+T) \text{ where }
\\
&
P:=\frac{2\binom{n}{3}}{\binom{N}{2}}
\\
&
Q:=
\frac{13-12(n-j)+3(n-j)^2}{\binom{N}{2}}
\\
&R:=
\frac{8\binom{n}{3}}{\binom{N}{2}}
\\
&S:=
\frac{6\binom{n}{4}}{\binom{n}{3}(N-2)}
\\& T:=
\frac{1}{(N-2)}\left(\frac{5n-11}{4}
\right)
\\
d_6(n,j):= &\  U+ 2(V+W+X) \text{ where }
\\&U:=\frac{n(n-1)(4n-5)}{6N^2}
\\
&V:=
4((n-j)-2)
\frac{n(n-1)(2n-1)}{6N^2}
\\
&W:=
\frac{3((n-j)-2) n (n-1)}{(4n-5)N}
\\\samepage
&Z:=\frac{2((n-j)-2)(2n-1)(n+1)}{(4n-5)N}
\\
D_5(n,j):=&\ \min\{d_5(n,j),2 \mu_5(n,j) d_5(n,j),1\}
\\
D_6(n,j):=&\ \min\{(d_6(n,j), 2\mu_6(n,j) d_6(n,j),1\}.
\end{align*}

The goal of this section is to give bounds on   $A_{5}(n,j)$ and $A_{6}(n,j)$ from below and above. They  are displayed in Proposition~\ref{Pro:ADS}, and the proofs of these bounds rely on Lemma~\ref{Le:ADS1} and Lemma~\ref{Le:ADS2}. 
In the sequel  when we write  $A_{5,6}$  we  denote two statements, 
one about $A_5$ and one about $A_6$.
Similarly for $D_{5,6}$ and~$\mu_{5,6}$

\begin{lemma}\label{Le:ADS1} 
For any $n\ge 3$ and $1\le j\le n-2$,
\[
 \frac{1}{{2}n^2}\le e^{-\mu_{5,6}(n,j)}\left(1+e^{\mu_{5,6}(n,j)}D_{5,6} (n,j)\right) \le 2.
\]
\end{lemma}

\begin{proof} 
By definition, $D_{5,6} (n,j)\le 1$,  and clearly $\mu_{5,6}(n,j)\ge 0$, then 
\[
 e^{-\mu_{5,6}(n,j)}\left(1+e^{\mu_{5,6}(n,j)}D_{5,6} (n,j)\right)=e^{-\mu_{5,6}(n,j)}+D_{5,6} (n,j)\le 2.
\]
To obtain a lower bound for {$e^{-\mu_{5,6}(n,j)}+D_{5,6} (n,j)$}, it suffices to bound the quantities $D_{5,6}(n,j)$.
\medskip

{\sl Lower bound for $D_{5}(n,j)$.} 
First we consider $\mu_5(n,j)$. The equality 
\begin{equation}\label{eq:binomN2}\binom{N}{2}=\frac{1}{2}\frac{n(n-1)}{2}\left(\frac{n(n-1)}{2}-1\right)=\frac{(n+1)n(n-1)(n-2)}{8}
\end{equation}
yields
\[
\mu_5(n,j)=\frac{2}{3}\frac{(n-j)(n-j-1)}{n+1}.
\]
This quantity,  $\mu_5(n,j)$, takes  its minimum when $j=n-2$. It follows that
\[\mu_5(n,j)\ge \frac{1}{n}.
\]
We claim that the quantity $d_5(n,j)=P+Q+(1-Q)((n-j)-2)(R+S+T)$ satisfies that 
\[d_5(n,j)\ge P \quad \text{for any} \quad 1\le j\le n-2.\]
  It is clear that $Q$ and $((n-j)-2)(R+S+T)$ are nonnegative. It only remains to prove that $1-Q$ is nonnegative. In fact, $Q\le 1/2$ for the values of $n$ and $j$ under consideration. 
 To prove that, first, we complete  squares and apply~\eqref{eq:binomN2}; then we take   $j=1$, and finally, we maximize over  over $n$ to obtain the last inequality, 
\[ 
Q= 8\frac{3((n-j)-2)^2+1}{(n+1)n(n-1)(n-2)}\le 8\frac{3(n-3)^2+1}{(n+1)n(n-1)(n-2)}\le \frac{1}{2}.
\]
\\
Finally, 
\[
d_{5}(n,j)\ge P=2\frac{\binom{n}{3}}{\binom{N}{2}}=\frac{8}{3(n+1)}\ge \frac{1}{n}.
\]
From the last lower bound and the bound $\mu_5(n,j)\ge 1/n$, we get the following: 
\[
D_5(n,j):= \min\left( d_{5}(n,j), 2\mu_5(n,j)d_5(n,j),1\right)\ge \min\left(\frac{1}{n} , \frac{2}{n^2},1\right)\ge\frac{1}{n^2}. 
\]  

{\sl Lower bound for $D_{6}(n,j)$.} First we consider $\mu_6(n,j)$. By definition,  $N=\binom{n}{2}$. 
\\It turns out  that
for every $n\ge 3$ and $j$ such that $1\le j \le n-2$,
\[
\mu_6(n,j)= \binom{(n-j)}{2} \frac{n(n-1)(4n-5)}{6 N^2}=\frac{1}{3}\frac{(n-j)(n-j-1)(4n-5)}{n(n-1)}\ge \frac{1}{n}.
\]
All the terms involved in the sum defining $d_6(n,j)$ are non-negative. 
Hence, 
\[
d_6(n,j){\ge U=} \frac{n(n-1)(4n-5)}{6N^2}=\frac{2}{3}\frac{4n-5}{n(n-1)}\ge \frac{2}{3n}.
\] 
Finally, the following holds and completes the proof of this lemma.
\[
D_6(n,j):= \min\left( d_{6}(n,j), 2\mu_6(n,j)d_6(n,j),1\right)\ge \min\left(\frac{2}{3n} , \frac{4}{3n^2},1\right)\ge\frac{1}{2n^2}. 
\]  
\end{proof}

 Let $\kappa$ be the map from $[0,1]$ to $\mathbb R$ given by 
\begin{equation}\label{kappa}
\kappa(x)=(e/2)^{1-x}(1-x)^{-(1-x)},\ \ \kappa(1)=1.
\end{equation}

\begin{lemma}\label{Le:ADS2} For any $n\ge 3$ and $j$ with 
$1\le j \le n-2$, the following holds
\begin{align}
 \label{Eq:LeA1}\frac{{e^{-2}}}{{2}\sqrt{n(n-1)}} \left(n^{1-j/n}\kappa(j/n)\right)^n&\le \binom{N}{n-j} \le    \left(n^{1-j/n}\kappa(j/n)\right)^n, 
 \\
  \label{Eq:LeA2} \frac{1}{4\sqrt{2\pi}\sqrt{n}} \left(n^{1-j/n}\kappa(j/n)\right)^n&\le \frac{N^{n-j}}{(n-j)!} \le \frac{ {1}}{\sqrt{2\pi}} \left(n^{1-j/n}\kappa(j/n)\right)^n. 
\end{align}
\end{lemma}

\begin{proof} 
{\sl We start by proving  Inequality~\eqref{Eq:LeA1}.}
With the bounds given for the binomial coefficients in~\eqref{binom}, the following holds
\begin{equation}\label{Eq:ini}
\frac{1}{2\sqrt N}\varphi((n-j)/N)^N\le \binom{N}{n-j}\le \varphi((n-j)/N)^N
\end{equation}
with $\varphi(x)=x^{-x}(1-x)^{-(1-x)}$, for any $n\ge 3$ and $1\le j\le n-2$. 
The expression  $\varphi((n-j)/N)^N$ has two factors, the first one  corresponds to $x^{-x}$ and the second one corresponds to $(1-x)^{1-x}$.
We replace $N$ by $n(n-1)/2$ only in the first factor. The  exponent of the second factor  is multiplied and divided by ${N}/({n-j})$. This leads to the following equality
\[
\varphi\left(\frac{(n-j)}{N}\right)^N
=
n^{n-j}\left(1-\frac{1}{n}\right)^{n-j}\left(2\left(1-\frac{j}{n}\right)\right)^{-(n-j)} b(n,j)
\]
with
\begin{equation}\label{Eq:bnj}
b(n,j)=\left(\left(1-\frac{n-j}{N}\right)^{\frac{N}{n-j}}\right)^{-(n-j)+\frac{(n-j)^2}{N}} .
\end{equation}
Let $c(n,j)$ be defined as
\[ c(n,j)=\left(1-\frac{1}{n}\right)^{n-j}b(n,j).\]
The right hand side of the equality before \eqref{Eq:bnj} is the product of four  factors. We leave the first and the third as they are. We deal with the second and the fourth. 
The factor $(1-1/n)^{n-j}$ satisfies 
\begin{equation}\label{Eq:1menos1n}e^{-1}
\le \left(1-\frac{1}{n}\right)^{n-1}\le \left(1-\frac{1}{n}\right)^{n-j}\le 1.
\end{equation}
The right hand side inequality is due to the fact that $1-1/n\le 1$. The left hand side inequality is due to the fact that $(1-{1}/{n})^{n-1}$  decreases towards its limit as $n\to \infty$. 

We  study  $b(n,j)$, defined in~\eqref{Eq:bnj}. 
First, we use the classical inequality 
\[-x-x^2\le \ln(1-x)\le -x \quad (0< x\le 2/3).
\]
After multiplying by $1/x$ and taking 
powers, 
we get
\begin{equation}\label{Eq:eala}
e^{-1-x}\le (1-x)^{1/x}\le e^{-1} \quad (0< x\le 2/3).
\end{equation}
Observe that, for $j\ge 1$,
\[
\frac{n-j}{N}=\frac{2(n-j)}{n(n-1)}\le\frac{2}{n}.
\]
Notice that  $0<({n-j})/{N}\le 2/3$ since $n\ge 3$. This allows us to replace $x$ by $(n-j)/N$ in~\eqref{Eq:eala}. It turns out that
\[
e^{-1-(\frac{n-j}{N})}\le \left(1-\frac{n-j}{N}\right)^{\frac{N}{n-j}} \le e^{-1}.
\]
To obtain $b(n,j)$, consider the previous expressions to the power $-(n-j)+ {(n-j)^2}/{N} $. With our bound on $(n-j)/N$, the exponent of the left hand side satisfies
\[
\left(-1-\frac{n-j}{N}\right)\left(-(n-j)+\frac{(n-j)^2}{N}\right)=(n-j)-\frac{(n-j)^3}{N^2}  
\ge (n-j)-1.
\]
Finally,
\begin{equation}\label{Eq:LemAf}e^{-1}e^{(n-j)}\le b(n,j)\le e^{(n-j)-\frac{(n-j)^2}{N}}\le e^{(n-j)}. 
\end{equation}
From Inequalities~\eqref{Eq:1menos1n} and~\eqref{Eq:LemAf}, it follows that  $c(n,j)e^{-(n-j)}$ takes values in $[e^{-2},1]$ and the following holds   
\begin{align*}
\varphi\left(\frac{(n-j)}{N}\right)^N&=c(n,j)n^{n-j}\left(2\left(1-\frac{j}{n}\right)\right)^{-(n-j)}\\
&=c(n,j)e^{-(n-j)}\left(n^{1-j/n}\kappa(j/n)\right)^n.
\end{align*}
To end the proof of  Inequality~\eqref{Eq:LeA1}
consider Inequality~\eqref{Eq:ini} together with the fact that $N=n(n-1)/2$. 

{\sl Proof of Inequality \eqref{Eq:LeA2}.} Approximating  the factorial by~\eqref{StirlingRobbins};  extracting $n$ as a common factor  in $(n-j)^{n-j}$, in $(n-1)^{n-j}$, and in $(n-1)^{n-j}$; and writing the final expression as an $n$-th power (similarly to what  it is done in~\eqref{tb}), we get 
\begin{align*}
\frac{N^{n-j}}{(n-j)!}&\le \frac{n^{n-j}(n-1)^{n-j}}{2^{n-j}}\frac{e^{n-j}}{\sqrt{2\pi(n-j)}(n-j)^{n-j}}\\
&=\frac{(1-1/n)^{n-j}}{\sqrt{2\pi(n-j)}}\left(n^{1-j/n}\kappa(j/n)\right)^n.
\end{align*}
We obtain the lower bound similarly,
\[\frac{N^{n-j}}{(n-j)!}\ge e^{-1/12}\frac{(1-1/n)^{n-j}}{\sqrt{2\pi(n-j)}}\left(n^{1-j/n}\kappa(j/n)\right)^n.\]
Finally, with  Inequality~\eqref{Eq:1menos1n}, and since  $1\le n-j\le n$, we obtain the bounds 
\[
\frac{1}{4\sqrt{2\pi n}}\le \frac{e^{-1-1/12}}{ \sqrt{2\pi n}}\le \frac{(1-1/n)^{n-j}}{\sqrt{2\pi(n-j)}}\le \frac{1}{\sqrt{2\pi }}
\]
that  prove the estimates  on $ {N^{n-j}}/(n-j)!$. 
\end{proof}

The next Proposition~\ref{Pro:ADS} is a direct consequence of  Lemmas~\ref{Le:ADS1} and~\ref{Le:ADS2}.
Recall that $\kappa:[0,1]\to \mathbb R$  defined in~\eqref{kappa}, 
$\kappa(x)=(e/2)^{1-x}(1-x)^{-(1-x)},\ \ \kappa(1)=1$.

\begin{proposition}\label{Pro:ADS} 
For any $n\ge 3$ and $1\le j\le n-2$,
\begin{equation}\label{Arratia-grafico}
 \frac{e^{-2}}{4n^3} \left(n^{1-j/n}\kappa(j/n)\right)^n \le A_{5,6}(n,j)\le 2\left(n^{1-j/n}\kappa(j/n)\right)^n.
\end{equation}
\end{proposition}

\begin{proof}
Lemma~\ref{Le:ADS1} proves that, for any  $n\ge 3$ and $1\le j\le n-2$,
\[
 \frac{1}{2n^2}\le e^{-\mu_{5,6}(n,j)}\left(1+e^{\mu_{5,6}(n,j)}D_{5,6} (n,j)\right) \le 2.
\]
Then,
\begin{eqnarray*}
\frac{1}{{2}n^2} \binom{N}{n-j}  \le& A_5(n,j)&\le 2 \binom{N}{n-j},
\\
\frac{1}{2n^2} \frac{N^{n-j}}{(n-j)!}  \le& A_6(n,j)&\le 2 \frac{N^{n-j}}{(n-j)!} .
\end{eqnarray*}
Lemma~\ref{Le:ADS2} provides us bounds on the terms involving combinatorials and factorials and gives 
\begin{align*}
 \frac{e^{-2}}{2\sqrt{n(n-1)}} \left(n^{1-j/n}\kappa(j/n)\right)^n&\le \binom{N}{n-j} \le   \left(n^{1-j/n}\kappa(j/n)\right)^n, 
\\
  \frac{1}{4\sqrt{2\pi}\sqrt{n}} \left(n^{1-j/n}\kappa(j/n)\right)^n &\le \frac{N^{n-j}}{(n-j)!} \le \frac{1}{\sqrt{2\pi}} \left(n^{1-j/n}\kappa(j/n)\right)^n. \end{align*}
Finally,
 \begin{align*}
\frac{e^{-2}}{4n^3}  \left(n^{1-j/n}\kappa(j/n)\right)^n &\le A_5(n,j)\le 2  \left(n^{1-j/n}\kappa(j/n)\right)^n, 
\\
 \frac{1}{8\sqrt{2\pi}n^2\sqrt n}\left(n^{1-j/n}\kappa(j/n)\right)^n  & \le A_6(n,j)\le  \frac{2}{\sqrt{2\pi}}\left(n^{1-j/n}\kappa(j/n)\right)^n.
\end{align*}
Combining both inequalities, Proposition~\ref{Pro:ADS} follows. 
\end{proof}

\subsection{Bender's estimate}

The notation $r_n\sim s_n$ indicates that $\lim_{n\to \infty}r_n/s_n=1$ when $n\to \infty$.
Bender~\cite{Bender}  establishes that for any real number~$r$ such that 
$0<r<1/2$,  then 
\[
  S_{n}^{(j)}    \sim\frac{n!e^{-\alpha j}}{j!\rho^{n+1}(1+e^\alpha)\sigma \sqrt{2\pi n}} 
\]
uniformly for $j/n\in [r,1-r]$,
where  $\alpha$ is such that 
\begin{align*}
\frac{n}{j}&=(1+e^{\alpha})\ln(1+e^{-\alpha})
\end{align*}
and
\begin{align*}
\rho&=\ln(1+e^{-\alpha}), 
\\
 \sigma^2& =\left(\frac{j}{n}\right)^2\big(1-e^\alpha\ln(1+e^{-\alpha})\big).
\end{align*}
We  introduce two   functions  to describe  the behavior of 
$S_{n}^{(j)}$  in terms of~$j/n$ (see Fig. \ref{Fig:psimu}),

\begin{align}\label{psimu}
\psi:(0,1)\mapsto \mathbb R, && 
\psi(x)=&\frac{e^{-((1-x)+x\delta(x) )}}{x^x \ln(1+e^{-\delta(x)}) } 
\\ 
\mu:(0,1)\mapsto \mathbb R, && 
\mu(x)=&\left(x (1-e^{\delta(x)}\ln(1+e^{-\delta(x)}))\right)^{1/2}  \nonumber
\end{align}
where  $\delta:(0,1) \mapsto \mathbb R$  is defined by
\begin{equation}\label{alfa}
\delta^{-1} (y) = \frac{1}{(1+e^{y})\ln(1+e^{-y})}.
\end{equation} 
The next lemma rephrases Bender's estimate using $\psi(j/n)$ and $\mu(j/n)$.

\begin{lemma}\label{Bender} 
For any positive  real number $r$ such that  $0<r<1/2$ and for  any real number  $C>1$
 there exists an integer $n_0=n_0(r,C)\ge 2$ such that
for every integer $n\ge n_0$ and for every integer $j$ with $1\le j\le n-1$ and $j/n\in [r,1-r]$. 
\[
 \frac{{e^{-1/12}}}{C \sqrt{2\pi n}\mu(j/n)}\left(n^{1-j/n}\psi(j/n)\right)^n \le 
 S_n^{(j)}
\le 
\frac{e^{1/12}C }{\sqrt{2\pi n}\mu(j/n)} \left(n^{1-j/n}\psi(j/n)\right)^n.
\]
\end{lemma}

\begin{proof}
Observe that 
\[
(1+e^\alpha)\rho \sigma=\big(1-e^\alpha\ln(1+e^{-\alpha})\big)^{1/2}.
\] 
Thus, Bender's estimate implies that for any $r$ with $0<r<1/2$ and for any  $C>1$ there exists $n_0=n_0(r,C)$ 
such that for any pair of positive  integers $n,j$, with  $n\ge n_0$ and  $j/n\in [r,1-r]$,
\begin{equation}\label{BenderInequalities}
\frac{1}{C }  T_\alpha(n,j) \le S_{n}^{(j)}\le  C T_\alpha(n,j)
\end{equation}
where
\begin{equation*}\label{Talfanj}
 T_\alpha(n,j)= \frac{n!}{j!}\frac{e^{-\alpha j}}{\rho^{n} (1-e^\alpha\ln(1+e^{-\alpha}))^{1/2} \sqrt{2\pi n}}.
\end{equation*}
Using~\eqref{StirlingRobbins} we  have
\[
{e^{-1/12}}e^{j-n}\frac{\sqrt{n}}{\sqrt{j}}  \frac{n^n}{j^{j}}\le \frac{n!}{j!}\le e^{1/12} e^{j-n}\frac{\sqrt{n}}{\sqrt{j}}  \frac{n^n}{j^{j}}.
\]
We remark that
\[e^{j-n}\frac{n^n}{j^j}=\left(e^{-(1-j/n)}\left(j/n\right)^{-j/n}\right)^n.
\]
Then, using the expressions for $\psi(n/j)$ and $\mu(j/n)$,
\[
\frac{e^{-1/12}}{\sqrt{2\pi n}\mu(j/n)}\left(n^{\frac{(n-j)}{n}}\psi(j/n)\right)^n \le T_\alpha(n,j)
\le \frac{e^{1/12}}{\sqrt{2\pi n}\mu(j/n)} \left(n^{\frac{(n-j)}{n}}\psi(j/n)\right)^n.
\]
Combining these inequalities with~\eqref{BenderInequalities} we obtain the wanted result.
 \end{proof}

\begin{figure}
\begin{center}
\includegraphics{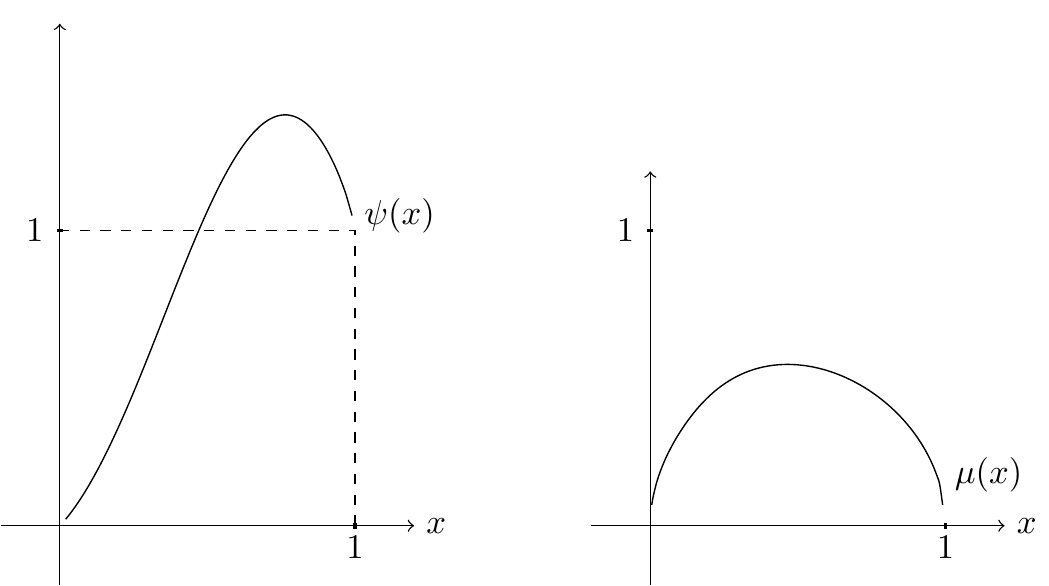}
\end{center}

\caption{Graphs of $\psi(x)$ and $\mu(x)$.}
\label{Fig:psimu}

\end{figure}

The functions $\psi(x)$ and $\mu(x)$ are  smooth and concave in the open interval~$(0,1)$.  
The function $\delta^{-1}(y)$ is increasing and  
\[
\lim_{x\to 0^+}\delta(x)=-\infty \hbox{ and }\lim_{x\to 1^-}\delta(x)=+\infty.
\]
From this, it is clear that  
$\lim_{x\to 0^+}\psi(x)=0$, $\lim_{x\to 1^-} \psi(x)=1$, and  $ \lim_{x\to 0^+} \mu(x)=$ $\lim_{x\to 1^-} \mu(x)=0$. 
Then,   the bounds given in Lemma~\ref{Bender} become indeterminate when $j/n$ 
is near $0$ or~$1$. This is why     $j/n $ must be in a central interval in~$(0,1)$.

The next corollary is a straightforward consequence of Lemma~\ref{Bender} 
and the fact that $\mu(x)$ is uniformly bounded on any closed interval included in $(0,1)$. 
The constants $c_1$ and $C_1$ in the statement of Corollary~\ref{corolarioBender}
 can be chosen as the minimum and maximum values of $\{\mu(x):x\in[r,1-r]\}$.

\begin{corollary}\label{corolarioBender}
For any positive  real number $r$ such that  $0<r<1/2$, 
there exist $c_1$ and $C_1$ such that 
for every pair of positive integers $n,j$ with  $j/n\in [r,1-r]$
we have
\begin{equation}\label{Bender-grafico}
 {e^{-1/12}}\frac{c_1}{\sqrt{2\pi n}}\left(n^{1-j/n}\psi(j/n)\right)^n \le  S_n^{(j)}\le e^{1/12}\frac{C_1}{\sqrt{2\pi n}} \left(n^{1-j/n}\psi(j/n)\right)^n.
\end{equation}
 \end{corollary}

\subsection{A plot}

The four upper bounds given in~\eqref{trivial},~\eqref{Rennie-grafico},~\eqref{Arratia-grafico} and~\eqref{Bender-grafico}
are of the form  
\[
S_n^{(j)}\leq  n^{n-j} \text{\tt bound}
\]
In order to visualize them we  divide both sides  by $n^{n-j} $ and we take   $n$-th root in both sides.
\[
\left(S_n^{(j)}/ n^{n-j}\right)^{1/n} \leq \text{\tt bound}^{1/n}
\]
In the four  cases 
{\tt  bound}$^{1/n}$ is of the form 
\[
\text{\tt expression}^{1/n} \left( f^n\right)^{1/n},
\]
where  {\tt expression}$^{1/n}$ goes to $1$ as $n$ goes to infinity
and $f$ is either $\theta$, $\eta$, $\kappa$ or $\psi$.
Thus,  we ignore  {\tt expression}$^{1/n}$.
   Figure~\ref{fig:comparison}  plots the following:

\begin{tabular}{ll}
\multicolumn{2}{l}{In dotted line, the exact value}
\\
&$\displaystyle
\widehat{S}_n^{(j)}=( {S}_n^{(j)}/n^{n-j})^{1/n}.$
\\\\
\multicolumn{2}{l}{The graphic of the function $\theta$ involved in the trivial bound}
\\
& $\displaystyle
\widehat{ S}_n^{(j)}\le \frac{1}{(\sqrt{2\pi j})^{1/n}}\theta(j/n),
\text{ where  $\theta(x)$ is given in~\eqref{theta}.}
$
\\\\
\multicolumn{2}{l}{The graphic of the function $\eta$ involved in  Rennie and Dobson's bound}
\\
&$\displaystyle
\widehat{ S}_n^{(j)}\le \frac{1}{2^{1/n}}\eta(j/n),
\text{ where  $\eta(x)$ is given in~\eqref{eta}.}
$
\\\\
\multicolumn{2}{l}{The graphic of the function $\kappa$ involved in Arratia and DeSalvo's bound}
\\
&$\displaystyle
\widehat{ S}_n^{(j)} \le {2}^{1/n}\kappa(j/n), 
\text{ where  $\kappa(x)$ is given in~\eqref{kappa}.}
$
\\\\
\multicolumn{2}{l}{In stroke gray line,  Bender's estimate}
\\
&$\displaystyle
\widehat{ S}_n^{(j)} \le \left(e^{1/12}\frac{C_1}{\sqrt{2\pi n}}\right)^{1/n}\psi(j/n),$
\\ 
&{where  $\psi(x)$ is given in~\eqref{psimu}
 and $C_1$ in Corollary~\ref{corolarioBender},}
\\
&with 
$j/n\in [r,1-r]$
for any real $r$ such that  $0<r<1/2$. \\
&The   constant $C_1$ depends on $r$. In the plot of Figure~\ref{fig:comparison}, $r=0.1$.  
\end{tabular}

\begin{figure}[h]
\begin{center}
\includegraphics{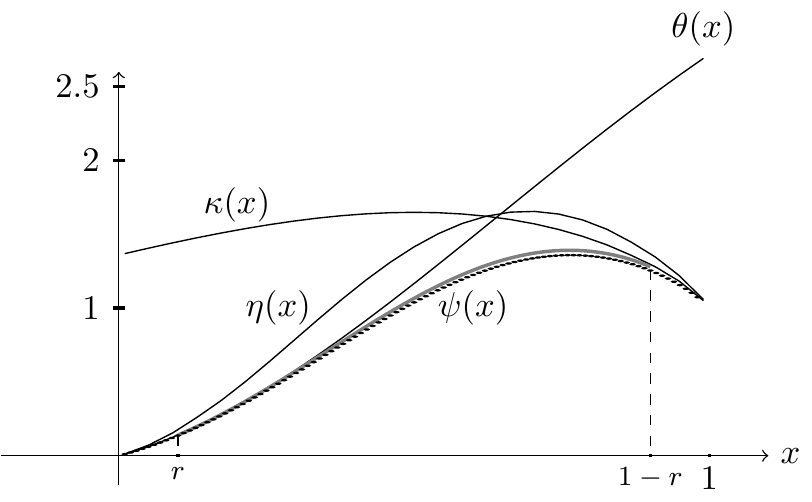}
\end{center}

\caption{Comparison of the normalized Stirling numbers of the second kind $ \widehat{S}_n^{(j)}=( {S}_n^{(j)}/n^{n-j})^{1/n}$ (in dotted line), for $n=100$ and $j=1,..,100$, with the four estimates (in solid lines).
}
\label{fig:comparison}

\end{figure}

\section{Application to our problem}

For the proof of Theorem~\ref{thm:a(n,j)} we must give an upper bounds of 
$a(n,j)$, which is always a positive term.
Since $a(n,j) =   
\binom{n}{j} j!\  S_n^{(j)}$,
 we can use  upper bounds for the  Stirling numbers of the second kind. 
We  choose   Rennie and Dobson's bound in the case $ j/n$ is near $0$ or $1$,
and the bound originated in Bender's estimate when $ j/n$ is in $ [1/r, 1-1/r]$, for $ r>0$.  

\subsection{When the ratio $j/n$ is  near $0$ or $1$}
The next lemma expresses this bound in terms of the ratio $j/n$ with the help of the function
\begin{equation}\label{funcionRennie}
\nu:[0,1]\to \mathbb R,\quad \nu(x)=x \ e^{-x}\varphi(x)^2,
\end{equation}
where   $\varphi(x)$ is defined  in~\eqref{combinatorial}. 

\begin{lemma}\label{nu1} 
For any pair of positive integers $n,j$ such that $n\ge 1$ and $1\le j\le n-1$,
 \[
 a(n,j) n^{-n}  \le     \sqrt{j}\, \nu(j/n)^n.
 \]
\end{lemma}
\begin{proof}
Recall that $a(n,j)=
\binom{n}{j} j!\  S_n^{(j)}$.
Rennie and Dobson's  upper bound~\eqref{Rennie}  for $S_n^{(j)}$ yields
\[
 a(n,j)\le   \frac{1}{2} \binom{n}{j}^2 j! j^{n-j}.
\]
 We apply the estimates~\eqref{StirlingRobbins}  for the factorial. Then we use the  upper bound for the binomial coefficient given in~\eqref{binom}  that involves the constant $C_0=e^{1/12}(\sqrt{2\pi})^{-1}$, which yields
  \begin{align*}
    \frac{1}{2} \binom{n}{j}^2 j! j^{n-j}\le& \frac{1}{2}(\sqrt 2 C_0)^2e^{1/12}\sqrt{2\pi} \sqrt{j}\varphi(j/n)^2e^{-j}j^n
    \\
    \le &\frac{e^{1/4}}{\sqrt{2\pi}} \sqrt{j}\, \nu(j/n)^n \\
    \le & \sqrt{j}\, \nu(j/n)^n. 
     \end{align*}
      \end{proof}
The function  $\nu(x)$ is  smooth and  concave, $\nu(0)=0$,
and $\nu(1)=e^{-1}$.  The bound given in  Lemma~\ref{nu1} is tight when  $j/n$ is near $0$ or~$1$. 
However, it is not good when $j/n$ takes values in middle of the interval $[0,1]$. 
In fact, this bound  
satisfies $\sqrt{j}(\nu(1/2))^n\ge  \sqrt{j} (1.1)^n>1$
but we know that $n^{-n}a(n,j)\le 1$  for any choice of $j$ and $n$.
This leads us to consider the only two real numbers $x_0$ and $x_1$  in $[0,1]$ for which $\nu(x_0)=\nu(x_1)=1$ and $x_0<x_1$. 
These numbers are $x_0\approx 0.387$ and $x_1\approx 0.790$. Figure~\ref{Fig:nu varphi} displays the graphs of $\nu(x)$ and $\varphi(x)$. 

\begin{figure}
\begin{center}
\begin{tikzpicture}[scale=2]
\draw[->] (-0.2,0) -- (1.2,0) node[right] {$x$};
\draw[->] (0,-0.2) -- (0,1.4) node[above] {};
\draw[line width=1pt] (1,-.01)--(1,.01)node[below] {$1$};
\draw[line width=1pt] (.01,1)--(-.01,1)node[left] {$1$};
\draw[line width=1pt] (1,-.01)--(1,.01)node[below] {};
\draw[dashed] (1,0.01)--(1,1)node[below] {};
\draw[dashed] (1,1)--(0,1)node[below] {};
\draw[domain=.01:.99] plot (\x,{exp(-\x)*\x*(\x^(-2*\x)*(1-\x)^(2*(\x-1))})
 node[right] {$\nu(x)$};
\end{tikzpicture}\hspace{1cm}
\begin{tikzpicture}[scale=2.5]
\draw[->] (-0.2,0) -- (1.2,0) node[right] {$x$};
\draw[->] (0,-0.2) -- (0,2.2) node[above] {};
\draw[line width=1pt] (1,-.01)--(1,.01)node[below] {$1$};
\draw[line width=1pt] (.01,1)--(-.01,1)node[left] {$1$};
\draw[line width=1pt] (.01,2)--(-.01,2)node[left] {$2$};
\draw[line width=1pt] (1,-.01)--(1,.01)node[below] {};
\draw[dashed] (1,0.01)--(1,2)node[below] {};
\draw[dashed] (1,2)--(0,2)node[below] {};
\draw[domain=.001:.999] plot (\x,{(\x^(-\x)*(1-\x)^((\x-1))})
 node[right] {$\varphi(x)$};
\end{tikzpicture}
\end{center}
\caption{Graphs of functions $\nu(x)$ and $\varphi(x)$.}
\label{Fig:nu varphi}
\end{figure}
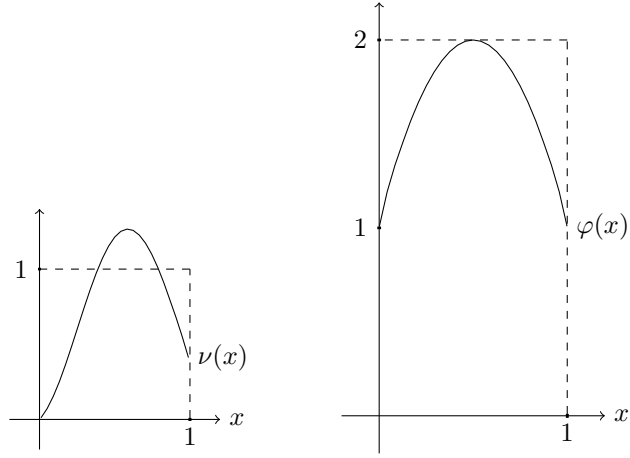

 \begin{lemma}\label{colitas} 
Let $x_0$ and  $x_1$ be such   that  $0<x_0<x_1<1$   and $\nu(x_0)=\nu(x_1)=1$.
For any pair of real numbers $r_0$ and $r_1$ such that   $0<r_0<x_0$ and $x_1<r_1<1$
 there exists a  real  number $\Lambda$ less than $1$,
such that for every positive integer $n$, 
 \[
n^{-n} a(n,j)\le 
 \sqrt{ n}\Lambda^{n} \hbox{, if }   j/n\in [0,r_0]\cup [r_1,1]. 
\]
\end{lemma}
\begin{proof}
Lemma~\ref{nu1} says that $a(n,j)n^{-n}\le \sqrt{j}\nu(j/n)^n$. 
The function  $\nu(x)$ is   smooth and concave  with $\nu(0)=0$, and $\nu(1)=e^{-1}$. 
 This implies the existence of unique points $x_0$ and $x_1$ 
 such   that  $0<x_0<x_1<1$   and $\nu(x_0)=\nu(x_1)=1$.
Fix $r_0$ and $r_1$ such that   $0<r_0<x_0$ and $x_1<r_1<1$.
Necessarily,  $\nu(r_0)<1$ and $\nu(r_1)<1$.  
Let $\Lambda_0=\nu(r_0)$ and $\Lambda_1=\nu(r_1)$. 
If $j/n \in[0,r_0]$ then
\[
  \nu(j/n)\le \max\{\nu(x): x\in [0,r_0] \} \le \Lambda_0 . 
\]
Similarly, if $j/n \in [r_1,1]$, we have $\nu(j/n)\le \Lambda_1$. 
Taking $\Lambda=\max\{\Lambda_0,\Lambda_1\}$, 
the lemma is proved.
\end{proof}

\paragraph{Example:} The choice $r_0=0.1$  yields   $\Lambda_0\approx 0.173 $, 
and $r_1=0.9$ yields $\Lambda_1\approx 0.701 $. In Figure~\ref{Fig:Thm1}, 
the value of $\Lambda$ equals the maximum between the approximations of $\Lambda_0$ and $\Lambda_1$.

\subsection{When the ratio $j/n$ is not near $0$ nor $1$}

We introduce the function 
\begin{equation}\label{phi}
\phi:(0,1)\mapsto \mathbb R, \ \ \ \   \phi(x)=(e\ln(1+e^{-\delta(x)}))^{-1}\varphi(x)e^{-x\delta(x)}
\end{equation}
where $\varphi(x)$ is defined in~\eqref{combinatorial} and $\delta(x)$ is defined in ~\eqref{alfa}. 

\begin{lemma}  \label{lemma:a(n,j)}
Consider the  constants $c_1$ and $C_1$  in  Corollary~\ref{corolarioBender}. 
For any real number $r$ such that $0<r<1/2$,  
and for any pair of positive integers $n,j$ such that  $j/n\in [r,1-r]$
\[
 \frac{{e^{-1/6}}c_1}{\sqrt{2\pi (n-j)}}\phi(j/n)^n \le  n^{-n} a(n,j)\le \frac{e^{1/6}C_1}{\sqrt{2\pi (n-j)}} \phi(j/n)^n .
\]
\end{lemma}
\begin{proof}
Write $a(n,j)=    S_n^{(j)} n!/(n-j)!$, 
then use  Stirling estimates ~\eqref{StirlingRobbins} for the factorial, apply Corollary~\ref{corolarioBender}
and use the definition of $\varphi(x)$ given in~\eqref{combinatorial}.
\end{proof}

 The function $\phi(x)$ is displayed in Figure~\ref{Fig:Thm1}. It  is smooth,  concave, $\phi(0)=0$ and  $\phi(1)=e^{-1}$.   
 The auxiliary function $\delta(x)$ takes the value $-\ln(e-1)$ at $x=1-1/e$ and then, $\phi(1-1/e)=1$. 
This value is the maximum of $\phi(x)$ because the lower bound given in Lemma~\ref{lemma:a(n,j)}
 implies that $\phi(x)\le 1$ for~$x\in (0,1)$.

\subsection{Proofs of Theorem~\ref{thm:a(n,j)} and Corollary~\ref{cor:sum}}

 Theorem~\ref{thm:a(n,j)}   considers the ratio between $j$ and~$n$.
The proof combines the two cases we just  studied: 
 when $j/n$ is near  $0$ or~$1$, and 
 when $j/n$ is in a central  interval away from $0$ and~$1$.

\begin{proof}[Proof of Theorem~\ref{thm:a(n,j)}] 
Consider the function $\nu$ given in \eqref{funcionRennie}.
Pick numbers $x_0$ and $x_1$ such that $0 < x_0 < x_1 < 1$ and  $\nu(x_0)=\nu(x_1)=1$.
 Take any $r\in (0,1/2)$ so that 
$ r \le \max \{x_0, 1-x_1\}$. If  $j/n\in[r,1-r]$
apply Lemma~\ref{lemma:a(n,j)}. 
Otherwise, apply  
Lemma~\ref{colitas}.
\end{proof}

The proof of  Corollary~\ref{cor:sum} is immediate from the statement of Theorem~\ref{thm:a(n,j)}.

\begin{proof}[Proof of Corollary~\ref{cor:sum}]
The result is a direct application of  Theorem~\ref{thm:a(n,j)}
because 
 \[
\max\{n^{-n} a(n,j), 1\le j\le \ell\}\le n^{-n}\sum_{j=1}^\ell a(n,j)\le n \max\{n^{-n} a(n,j), 1\le j\le \ell\}.
\]
  \end{proof}

\noindent
{\bf Acknowledgements}.  We thank an anonymous referee  for noticing an error in  the lower bound  given in Theorem 1 in a previous version of this paper and for many  valuable suggestions.

\end{document}